\documentclass[a4paper,12pt]{article}
\usepackage{a4wide}
\usepackage{amsmath}
\usepackage{amssymb}
\usepackage{amsthm}
\usepackage{latexsym}
\usepackage{graphicx}
\usepackage[english]{babel}
\usepackage{makeidx}

\newtheorem{obs} [subsection]{Remark}

\newtheorem{prop}[subsection]{Proposition}

\newtheorem{teor}[subsection]{Theorem}
\newtheorem{lema}[subsection]{Lemma}
\newtheorem{cor} [subsection]{Corollary}
\newcommand{\Zng}{$\mathbb Z^n$-graded $S$-module}
\newcommand{\Zngp}{$\mathbb Z^{n}$-graded $S'$-modules}

\def\sdepth{\operatorname{sdepth}}
\def\depth{\operatorname{depth}}
\def\supp{\operatorname{supp}}

\begin{document}
\selectlanguage{english}
\frenchspacing

\large
\begin{center}
\textbf{Stanley depth of quotient of monomial complete intersection ideals}

Mircea Cimpoea\c s
\end{center}
\normalsize

\begin{abstract}
We compute the Stanley depth for a particular, but important case, of the quotient of complete intersection monomial ideals. Also, in the general case, we give sharp bounds for the Stanley depth of a quotient of complete intersection monomial ideals. In particular, we prove the Stanley conjecture for quotients of complete intersection monomial ideals.


\end{abstract}

\section*{Introduction}

Let $K$ be a field and $S=K[x_1,\ldots,x_n]$ the polynomial ring over $K$.
Let $M$ be a \Zng. A \emph{Stanley decomposition} of $M$ is a direct sum $\mathcal D: M = \bigoplus_{i=1}^rm_i K[Z_i]$ as a $\mathbb Z^n$-graded $K$-vector space, where $m_i\in M$ is homogeneous with respect to $\mathbb Z^n$-grading, $Z_i\subset\{x_1,\ldots,x_n\}$ such that $m_i K[Z_i] = \{um_i:\; u\in K[Z_i] \}\subset M$ is a free $K[Z_i]$-submodule of $M$. We define $\sdepth(\mathcal D)=\min_{i=1,\ldots,r} |Z_i|$ and $\sdepth_S(M)=\max\{\sdepth(\mathcal D)|\;\mathcal D$ is a Stanley decomposition of $M\}$. The number $\sdepth_S(M)$ is called the \emph{Stanley depth} of $M$. Stanley \cite{stan} conjectured that $\sdepth_S(M)\geq\depth_S(M)$ for any \Zng $\;M$. Herzog, Vladoiu and Zheng show in \cite{hvz} that $\sdepth_S(M)$ can be computed in a finite number of steps if $M=I/J$, where $J\subset I\subset S$ are monomial ideals. However, it is difficult to compute this invariant, even in some very particular cases.

In section $1$, we consider the case of quotients of irreducible monomial ideals, and we compute their Stanley depth, see Theorem $1.5$. In section $2$, we consider the general case of two complete intersection monomial ideals $I\subset J \subset S$. In Theorem $2.4$ we give sharp bounds for $\sdepth_S(J/I)$. Remark $2.10$ shows that these bounds are best possible. However, in a particular case, we give an explicit formula for $\sdepth_S(J/I)$, see Theorem $2.9$.

\footnotetext[1]{The support from grant ID-PCE-2011-1023 of Romanian Ministry of Education, Research and
Innovation is gratefully acknowledged.}

\section{The case of irreducible ideals}

First, we give the following technical Lemma.

\begin{lema}
Let $b$ be a positive integer, denote $S'=K[x_2,\ldots,x_n]$ and let $I\subsetneq J\subset S'$ be two monomial ideals.
Then $(x_1^b,J)/(x_1^b,I) \cong \bigoplus_{i=0}^{b-1} x_1^i(J/I)$, as \Zngp.
Moreover, $\sdepth_S((x_1^b,J)/(x_1^b,I)) = \sdepth_{S'}(J/I)$.
\end{lema}

\begin{proof}
Let $u\in (x_1^b,J)\setminus (x_1^b,I)$ be a monomial. Then $u=x_1^i\cdot u'$, for some nonnegative integer $i$ and some monomial $u'\in S'$. Since $u\notin (x_1^b,I)$, it follows that $i<b$ and, also, $u'\notin I$. On the other hand, since $u\in (x_1^b,J)$ and $u\notin x_1^bS$, it follows that $u'\in J$. Therefore, $u\in x_1^i(J\setminus I)$. Conversely, if we take a monomial $u'\in J\setminus I$ and an integer $0\leq i <b$, one can easily see that $u:=x_1^i\cdot u'\in (x_1^b,J)\setminus (x_1^b,I)$.

The decomposition $(x_1^b,J)/(x_1^b,I) \cong \bigoplus_{i=0}^{b-1} x_1^i(J/I)$ implies $\sdepth_S((x_1^b,J)/(x_1^b,I)) \geq \sdepth_{S'}(J/I)$ and, also, $J/I = ((x_1^b,J)/(x_1^b,I))\cap (S'/I)$, via the natural injection $S'/I \hookrightarrow S/(x_1^b,J)$. In order to prove the other inequality, we consider a Stanley decomposition of $(x_1^b,J)/(x_1^b,I) = \bigoplus_{i=1}^r v_i K[Z_i]$. Note that $v_i K[Z_i] \cap S' = \{0\}$ if $x_1|u$ and, otherwise, $v_i K[Z_i] \cap S' = v_i K[Z_i]$. Therefore, $J/I=\bigoplus_{x_1\nmid v_i}v_i K[Z_i]$ and thus $\sdepth_{S'}(J/I)\geq \sdepth_S((x_1^b,J)/(x_1^b,I))$ as required.
\end{proof}

An easy corollary of Lemma $1.1$ is the following.

\begin{cor}
Let $0\leq m < n$ be an integer. Then, $$\sdepth_S((x_1,\ldots,x_n)/(x_1,\ldots,x_m))= n - m - \left\lfloor \frac{n-m}{2} \right\rfloor.$$
\end{cor}

\begin{proof}
We use induction on $m$. If $m=0$, then, by \cite[Theorem 1.1]{par}, $\sdepth_S((x_1,\ldots,x_n))=\left\lceil n/2 \right\rceil = n - \left\lfloor \frac{n}{2} \right\rfloor$, as required. The case $m=n$ is trivial.

Now, assume $1\leq m < n$. Let $S'=K[x_2,\ldots,x_n]$ and denote $J=(x_2,\ldots,x_n)\subset S'$ and $I=(x_2,\ldots,x_m) \subset S'$. According to Lemma $1.1$, $(x_1,J)/(x_1,I)\cong J/I$ and thus, by induction hypothesis, 
$\sdepth_S((x_1,\ldots,x_n)/(x_1,\ldots,x_m)) = \sdepth_{S'}(J/I) = (n - 1) - (m-1) - \left\lfloor \frac{(n-1)-(m-1)}{2} \right\rfloor = n - m - \left\lfloor \frac{n-m}{2} \right\rfloor$, which complete the proof.
\end{proof}

If we denote $S''=K[x_{m+1},\ldots,x_n]$, note that the above Corollary follows also from the isomorphism of multigraded $S''$-modules, $(x_1,\ldots,x_n)/(x_1,\ldots,x_m) \cong (x_{m+1},\ldots,x_n)$.

\begin{lema}
Let $1\leq a<b$ be two integers, denote $S'=K[x_2,\ldots,x_n]$ and let $J\subset S'$ be a monomial ideals.
Then $(x_1^a,J)/(x_1^b,J) \cong \bigoplus_{i=a}^{b-1} x_1^i(S'/J)$, as \Zngp.
Moreover, $\sdepth_S((x_1^a,J)/(x_1^b,J)) = \sdepth_{S'}(S'/J)$.
\end{lema}

\begin{proof}
The proof is similar to the proof of Lemma $1.1$. In order to prove the $\mathbb Z^n$-graded isomorphism $(x_1^a,J)/(x_1^b,J) \cong \bigoplus_{i=a}^{b-1} x_1^i(S'/J)$, it is enough to see that a monomial $u\in (x_1^a,J)\setminus (x_1^b,J)$ if and only if $u=x_1^i\cdot u'$, for some integer $a\leq i <b$ and some monomial $u'\in S\setminus J$. The isomorphism implies $\sdepth_S((x_1^a,J)/(x_1^b,J)) \geq \sdepth_{S'}(S'/J)$ and in order to prove the other inequality, as in Lemma $1.1$, it is enough to note that $x_1^a(S'/J) = ((x_1^a,J)/(x_1^b,J)) \cap x_1^a (S'/J)$.
\end{proof}

\begin{lema}
Let $1\leq a < b$ be two integers and denote $S':=K[x_2,\ldots,x_n]$. Let $I\subset J\subset S'$ be two monomial ideals. Then:


$\sdepth_S((x_1^a,J)/(x_1^b,I))\geq \min\{\sdepth_{S'}(J/I),  \sdepth_{S'}(S'/I) \}$.

\end{lema}

\begin{proof}
Note that $(x_1^a,J)/(x_1^b,I) \cong (x_1^a,J)/(x_1^a,I) \oplus (x_1^a,I)/(x_1^b,I)$ as \Zngp. Using Lemma $1.1$ and Lemma $1.3$ we get the required result.
\end{proof}

\begin{teor}
Let $0\leq m \leq n$ be two integers. Let $a_i\geq 1$, for $1\leq i\leq n$ and $b_i\geq a_i$, for $1\leq i\leq m$, be some integers. Then $\sdepth((x_1^{a_1},\ldots,x_n^{a_n})/(x_1^{b_1},\ldots,x_m^{b_m}))=n-m-\left\lfloor \frac{n-m}{2} \right\rfloor$ and thus, $\sdepth_S((x_1^{a_1},\ldots,x_n^{a_n})/(x_1^{b_1},\ldots,x_m^{b_m}))\geq \depth_S((x_1^{a_1},\ldots,x_n^{a_n})/(x_1^{b_1},\ldots,x_m^{b_m}))$.
\end{teor}

\begin{proof}
We use induction on $m$. If $m=0$, by \cite[Theorem 1.3]{mirc}, $\sdepth_S((x_1^{a_1},\ldots,x_n^{a_n}))=\left\lceil \frac{n}{2} \right\rceil = n - \left\lfloor  \frac{n}{2} \right\rfloor$, as required. If $m=n$, then $(x_1^{a_1},\ldots,x_n^{a_n})/(x_1^{b_1},\ldots,x_n^{b_m})$ is a finite $K$-vector space and thus its Stanley depth is $0$. 

Now, assume $1\leq m < n$. We denote $S'=K[x_2,\ldots,x_n]$, $J=(x_2^{a_2},\ldots,x_n^{a_n})\subset S'$ and $I=(x_2^{b_2},\ldots,x_m^{b_m})\subset S'$. By induction hypothesis, we have $\sdepth_{S'}(J/I)=n-1-(m-1) - \left\lfloor \frac{n-m}{2} \right\rfloor = n - m - \left\lfloor \frac{n-m}{2} \right\rfloor$. On the other hand, by \cite[Theorem 1.1]{asia1} or \cite[Lemma 3.6]{hvz}, $\sdepth_{S'}(S'/I) = n-m$. Thus, according to Lemma $1.4$, we have $\sdepth_S((x_1^{a_1},\ldots,x_n^{a_n})/(x_1^{b_1},\ldots,x_m^{b_m}))\geq \sdepth_{S'}(J/I) =
n-m-\left\lfloor \frac{n-m}{2} \right\rfloor$.

If $a_1=b_1$, by Lemma $1.1$, we are done. Assume $a_1<b_1$. We denote $a=a_1$, $b=b_1$ and we consider the decomposition $(x_1^a,J)/(x_1^b,I) \cong \bigoplus_{i=0}^{a-1} x_1^i(J/I) \oplus \bigoplus_{i=a}^{b-1} x_1^i(S'/I)$ given by Lemma $1.4$. As in the proof of Lemma $1.1$, we consider a Stanley decomposition $(x_1^a,J)/(x_1^b,I) = \bigoplus_{j=1}^r v_jK[Z_j]$. It follows that $J/I = ((x_1^a,J)/(x_1^b,I))\cap (S'/I)$ and thus, $J/I=\bigoplus_{x_1\nmid v_i}v_i K[Z_i]$. 

In order to complete the proof, notice that $\depth((x_1^{a_1},\ldots,x_n^{a_n})/(x_1^{b_1},\ldots,x_m^{b_m})) = 1$ if $n>m$ and $0$, if $m=n$.
\end{proof}

\section{The case of complete intersection ideals}

\begin{lema}
Let $1\leq m<n$ be an integer, $I_1\subsetneq J_1\subset S':=K[x_1,\ldots,x_m]$ be two distinct monomial ideals and let $I\subset S''=K[x_{m+1},\ldots,x_n]$ be a monomial ideal. Then 
\[ \sdepth_S \frac{(J_1,I)}{(I_1,I)} \geq \sdepth_{S''} \left( \frac{S''}{I} \right) + \sdepth_{S'} \left( \frac{J_1}{I_1} \right).\]
\end{lema}

\begin{proof}
Let $u\in (J_1,I)\setminus (I_1,I)$ be a monomial. We write $u=u'\cdot u''$, where $u'\in S'$ and $u''\in S''$. Since $u\in (J_1,I)$, it follows that $u\in J_1S$ or $u\in IS$. On the other hand, since $u\notin (I_1,I)$, it follows that $u\notin I_1S$ and $u\notin IS$. Therefore, we get $u\in J_1S$ and so $u'\in J_1$. Also, since $u\notin IS$, it follows that $u''\notin I$. Similarly, we have $u'\notin I_1$. Thus, $u$ is a product between a monomial from $J_1\setminus I_1$ and a monomial from $S''\setminus I$. Also, if we take two arbitrary monomials $u'\in J_1\setminus I_1$ and $u'' \in S''\setminus I_1$, one can easily check that $u'\cdot u'' \in (J_1,I)\setminus (I_1,I)$. 

In consequence, given two Stanley decompositions $J_1/I_1 = \bigoplus_{i=1}^r u_iK[Z_i]$ and $S''/I = \bigoplus_{j=1}^s v_jK[Y_j]$, it follows that $(J_1,I)/(I_1,I) = \bigoplus_{i=1}^r \bigoplus_{j=1}^s u_iv_jK[Z_i\cup Y_j]$ is a Stanley decomposition and, thus, $\sdepth_S((J_1,I)/(I_1,I))\geq \sdepth_{S'}(J_1/I_1) + \sdepth_{S''}(S''/I)$.
\end{proof}

\begin{lema}
Let $1\leq m<n$ be an integer, $I_1\subsetneq J_1\subset S':=K[x_1,\ldots,x_m]$ be two monomial ideals and 
$I_2\subset J_2\subsetneq S'':=K[x_{m+1},\ldots,x_n]$ be other monomial ideals. Then:
\[ \sdepth_S \frac{(J_1,J_2)}{(I_1,I_2)} \geq \min \{ \sdepth_{S'} \left( \frac{S'}{J_1} \right)+\sdepth_{S''}\left(\frac{J_2}{I_2} \right), 
\sdepth_{S''} \left( \frac{S''}{I_2} \right) +\sdepth_{S'} \left( \frac{J_1}{I_1} \right) \}. \]
\end{lema}

\begin{proof}
We apply \cite[Lemma 2.1]{okazaki} to the short exact sequence
\[ 0 \longrightarrow (J_1,I_2)/(I_1,I_2) \longrightarrow (J_1,J_2)/(I_1,I_2) \longrightarrow (J_1,J_2)/(J_1,I_2) \longrightarrow 0, \]
and thus we are done by Lemma $2.1$.
\end{proof}

If $u\in S$ is a monomial, denote $\supp(u)=\{x_i:\; x_i|u\}$ the \emph{support} of the monomial $u$.

\begin{lema}
Let $u_1,\ldots,u_m \in S$ and $v_1,\ldots,v_m \in S$ be two regular sequence of monomials, such that $u_i|v_i$ and $v_j\neq u_j$ for some index $j$. Then $\sdepth_S((u_1,\ldots,u_m)/(v_1,\ldots,v_m)) = n-m$.
Moreover, $(u_1,\ldots,u_m)/(v_1,\ldots,v_m)$ has a Stanley decomposition with all its Stanley spaces of dimension $n-m$.
\end{lema}

\begin{proof}
We use induction on $m\geq 1$. If $m=1$, then $(u_1)/(v_1)\cong S/(v_1/u_1)$ and therefore $\sdepth_S((u_1)/(v_1))=n-1$, by \cite[Theorem $1.1$]{asia1}. Assume $m>1$. We apply Lemma $2.2$ for $J_1=(u_1,\ldots,u_{m-1})$, $J_2=(u_m)$, $I_1=(v_1,\ldots,v_{m-1})$ and $I_2=(v_m)$.

We get $\sdepth_S((u_1,\ldots,u_m)/(v_1,\ldots,v_m)) \geq \min \{ \sdepth_S(S/J_1) - 1, \sdepth_S(S/I_2) - 1 \} = n-m+1-1 = n-m$. In order to prove the opposite inequality, let $uK[Z]$ be a s Stanley space of $(u_1,\ldots,u_m)/(v_1,\ldots,v_m)$. Since $v_iS \cap u K[Z] = (0)$, it follows that there exists an index $j_i$ such that $x_{j_i}\notin Z$. Now, since the $v_1, \ldots v_m$ is a regular sequence, their supports are disjoint and therefore $\{x_{j_1},\ldots,x_{j_m}\}$ is a set of $m$ variables which do not belong to $Z$ and thus $|Z|\leq n-m$.
\end{proof}

Note that the $"\leq"$ follows also from the inequalities $\sdepth_S((u_1,\ldots,u_m)/(v_1,\ldots,v_m))\leq \dim((u_1,\ldots,u_m)/(v_1,\ldots,v_m)) \leq \dim(S/(v_1,\ldots,v_m)) = n - m$.

\begin{teor}
Let $I\subsetneq J\subset S$ be two monomial complete intersection ideals. Assume $J$ is generated by $q$ monomials and $I$ is generated by $p$ monomials. Then:
\[ n-p \geq \sdepth(J/I) \geq n-p - \left\lfloor \frac{q-p}{2} \right\rfloor. \]
\end{teor}

\begin{proof}
Assume $J = (u_1,\ldots,u_q)$ and $I=(v_1,\ldots,v_p)$. Since $v_1,\ldots,v_p$ is a regular sequence on $S$, their supports are disjoint. If we take a Stanley space $vK[Z]\subset J/I$ it follows, as in the proof of Lemma $2.3$ that $|Z|\leq n-p$, and thus $n-p \geq \sdepth_S(J/I)$.

Now, we prove the second inequality. If $p=0$, then, by \cite[Theorem 1.1]{shen}, we have $\sdepth_S(J/I) = \sdepth_S(J) = n - \left\lfloor \frac{q}{2} \right\rfloor$ and we are done. Also, in the case $p=q$, we are done by Lemma $2.3$. Assume $1\leq p < q$. Since $I\subset J$ we can assume that $u_1|v_1$. Note that $ u_1 \nmid v_j$ for all $j>1$. Indeed, if $p\geq 2$ and $u_1|v_2$, then $\supp(v_1)\cap\supp(v_2) \supseteq \supp(u_1)$, a contradiction with the fact that $v_1,v_2$ is a regular sequence. Thus, using induction, we may assume $v_i|u_i$ for all $1\leq i\leq p$. We denote $J_1=(u_1,\ldots,u_p)$ and $J_2=(u_{p+1},\ldots,u_q)$.

We use decomposition $(J_1,J_2)/I = (J_1,J_2)/J_1 \oplus J_1/I$.
Using \cite[Corollary 2.4(5)]{mirci} and \cite[Theorem 1.1]{asia1}, we get
$ \sdepth_S((J_1,J_2)/J_1)\geq \sdepth_S(S/J_1) - \left\lfloor \frac{q-p}{2} \right\rfloor = n - p - \left\lfloor \frac{q-p}{2} \right\rfloor$. 
On the other hand if $I\subsetneq J_1$, by Lemma $2.3$, $\sdepth_S(J_1/I)=n-p$. Thus, we get $\sdepth_S((J_1,J_2)/I)\geq n - p - \left\lfloor \frac{q-p}{2} \right\rfloor$, as required.
\end{proof}

\begin{cor}
With the notations of $2.4$, if $q=p+1$, then $\sdepth_S(J/I) = n-p$.
\end{cor}

\begin{cor}
If $J\subsetneq I\subset S$ are two monomial complete intersection ideals, then $\sdepth_S(J/I)\geq \depth_S(J/I)$.
\end{cor}

\begin{proof}
It is enough to notice that $\depth_S(J/I)=n-q+1$ if $q>p$, or $\depth_S(J/I)=n-q$ if $q=p$ and then apply Theorem $2.4$.
\end{proof}

\begin{lema}
Let $1\leq m<n$ be an integer, $I\subsetneq J\subset S':=K[x_1,\ldots,x_m]$ be two distinct monomial ideals and let $u \in S''=K[x_{m+1},\ldots,x_n]$ be a monomial. Then 
\[ \sdepth_S \frac{(J,u)}{(I,u)} = \sdepth_{S}\left(\frac{JS}{IS} \right) - 1.\]
\end{lema}

\begin{proof}
Since $\sdepth_{S''}(S''/(u)) = n-m-1$, the $"\geq"$ inequality follows by Lemma $2.1$. In order to prove the other inequality, let $(J,u)/(I,u)= \bigoplus_{j=1}^r v_jK[Z_j]$ be a Stanley decomposition, with its Stanley depth equal with $\sdepth_S((J,u)/(I,u))$. Note that $J/I = ((J,u)/(I,u))\cap (S'/I)$. Indeed, as in the proof of Lemma $1.1$, one can easily see that a monomial $v\in (J,u)\setminus (I,u)$ if and only if $v=v'\cdot v''$ for some monomials $v'\in J\setminus I$ and $v''\in S''\setminus (u)$. 

It follows that $J/I=(\bigoplus_{j=1}^r v_jK[Z_j]) \cap (S'/I) = \bigoplus_{j=1}^r (v_jK[Z_j]\cap (S'/I))$. If $v_j\notin S'$, obviously, then $v_jK[Z_j]\cap (S'/I)=\{0\}$. If $v_j\in S'$, then $v_jK[Z_j]\cap (S'/I) = v_jK[Z_j\setminus \{x_{m+1},\ldots,x_n\} ]$. Note that $\{x_{m+1},\ldots,x_n\} \nsubseteq Z_j$, because $uS\cap v_jK[Z_j] = \{0\}$. Thus, $|Z_j\setminus \{x_{m+1},\ldots,x_n\}|\geq |Z_j|-n+m+1$. Thus, we obtained a Stanley decomposition for $J/I$ with its Stanley depth $\geq \sdepth((J,u)/(I,u)) - n + m + 1$. It follows that $\sdepth_{S}(JS/IS)=\sdepth_{S'}(J/I)+n-m \geq \sdepth((J,u)/(I,u)) + 1$, as required.
\end{proof}

\begin{lema}
Let $1\leq m<n$ be an integer, $J\subset S':=K[x_1,\ldots,x_m]$ be a monomial ideal and let $u,v \in S''=K[x_{m+1},\ldots,x_n]$ be two distinct monomials with $u|v$. Then 
\[ \sdepth_S  \frac{(J,u)}{(J,v)}  = \sdepth_{S}\left( \frac{S}{JS} \right) - 1.\]
\end{lema}

\begin{proof}
As in the proof of the previous lemma, by $2.1$, we get the $"\geq"$ inequality. In order to prove the other inequality, note that $(J,u)\setminus (J,v) \cap u(S'/J) = u(S'/J)$. Indeed, a monomial $w\in (J,u)/(J,v)$ if and only if $w=w'\cdot w''$ for some monomials $w'\in S'\setminus J$ and $w''\in (u)\setminus (v)$. Using a similar argument as in the proof of Lemma $2.7$, we are done.
\end{proof}

Now, we are able to prove the following result, which generalizes Theorem $1.5$.

\begin{teor}
Let $u_1,\ldots,u_q\in S$ and $v_1,\ldots,v_p\in S$ be two regular sequences with $u_i|v_i$ for all $1\leq i\leq p$, where $q\geq p$ are positive integers. We consider the monomial ideals $J=(u_1,\ldots,u_q)\subset S$ and $I=(v_1,\ldots,v_p)$. We assume also that $u_{p+1},\ldots,u_q$ is a regular sequence on $S/I$.
Then $\sdepth(J/I) = n - p - \left\lfloor  \frac{q-p}{2} \right\rfloor$.
\end{teor}

\begin{proof}
We use induction on $p$. If $p=0$, by \cite{shen} and \cite[Lemma 3.6]{hvz}, we have $\sdepth_S(J/I) = n- \left\lfloor q/2 \right\rfloor$ as required. If $q=p$, by Lemma $2.4$, we get $\sdepth(J/I) = n-p$ and we are done.
Now, assume $1\leq p < q$. We denote $J_1=(u_2,\ldots,u_q)$ and $I_1=(v_2,\ldots,v_p)$. By induction hypothesis,
we have $\sdepth_S(J_1/I_1) = n - p + 1 - \left\lfloor  \frac{(q-1)-(p-1)}{2} \right\rfloor = n - p - \left\lfloor  \frac{q-p}{2} \right\rfloor + 1$. If $u_1=v_1$, by Lemma $2.7$, it follows that $\sdepth_S(J/I)=n-p -\left\lfloor  \frac{q-p}{2} \right\rfloor$, so we may assume $u_1\neq v_1$.
Either way, by Theorem $2.4$, we have $\sdepth(J/I)\geq n-p -\left\lfloor  \frac{q-p}{2} \right\rfloor$. Note that this inequality can be deduced also from Lemma $2.7$ and Lemma $2.8$, using the decomposition $J/I = (J_1,u_1)/(I_1,u_1)\oplus (I_1,u_1)/(I_1,v_1)$. 

In order to prove the other inequality, we consider a Stanley decomposition $J/I = \bigoplus_{j=1}^r w_j K[Z_j]$ with its Stanley depth equal to $\sdepth(J/I)$. Since, by hypothesis, $u_2,\ldots,u_q$ is a regular sequence on $S/(v_1)$, by reordering of variables, we may assume that $\supp(v_1)=\{x_{m+1},\ldots,x_n\}$ and $u_2,\ldots,u_q \in S':=K[x_1,\ldots,x_m]$, where $1\leq m<n$ is an integer. Thus $J_1$ and $I_1$ are the extension in $S$ for some monomial ideals in $\bar{J}_1,\bar{I}_1\subset S'$ generated by the same monomials as $J_1$ and $I_1$.
Note that $\bar{J}_1/\bar{I}_1 = (J/I)\cap (S'/\bar{I}_1) $ , where we regard $S'/\bar{I}_1$ as a submodule of $S/I$.
Using the same argument as in the proof of Lemma $2.7$, we get
$\sdepth_{S'}(\bar{J}_1/\bar{I}_1) \geq \sdepth_{S}(J/I) - n + m + 1$.
On the other hand, by induction hypothesis, we have
$\sdepth_{S'}(\bar{J}_1/\bar{I}_1) = m - p + 1 - \left\lfloor  \frac{q-p}{2} \right\rfloor$,
and thus we are done.
\end{proof}

\begin{obs}
Note that the hypothesis $u_{p+1},\ldots,u_q$ is a regular sequence on $S/I$ from the Theorem $2.9$ is essential in order to have the equality. Take for instance $J=(x_1,x_2,x_3)\subset S$ and $I=(x_1x_2x_3)$. Then $J/I = x_1K[x_1,x_2] \oplus x_2K[x_2,x_3] \oplus x_3K[x_1,x_3]$ is a Stanley decomposition for $J/I$ and therefore $\sdepth_S(J/I) = 2 > 3-1 - \left\lfloor \frac{3-1}{2} \right\rfloor = 1$.
\end{obs}

We end our paper with the following result, which generalize \cite[Proposition 2.7]{mirci} and \cite[Proposition 1.3]{pop}.

\begin{prop}
Let $I\subset J\subset S$ be two monomial ideals and let $u\in S$ be a monomial. Then, either $(I:u)=(J:u)$, either $\sdepth_S((J:u)/(I:u))\geq \sdepth_S(J/I)$.
\end{prop}

\begin{proof}
It is enough to consider the case $u=x_1$ and to assume that $(I:x_1)\subsetneq (J:x_1)$. Firstly, note that $x_1(I:x_1)=I\cap (x_1)$ and $x_1(J:x_1)=J\cap (x_1)$. Therefore, we have  
$ (J:x_1)/(I:x_1) \cong (J\cap(x_1))/(I\cap(x_1)) \cong (J/I)\cap (x_1)$, 
as multigraded $K$-vector spaces. 

Let $J/I = \bigoplus_{i=1}^r u_i K[Z_i]$ be a Stanley decomposition for $J/I$. It follows that
$(J/I)\cap (x_1) = \bigoplus_{i=1}^r (u_i K[Z_i] \cap x_1 S)$. One can easily see that, if $x_1\notin \supp(u_i)\cup Z_i$, then $u_i K[Z_i] \cap x_1 S = \{0\}$. Otherwise, we claim that $u_i K[Z_i] \cap x_1 S = LCM(u_i,x_1)K[Z_i]$. Indeed, if $x_1|u_i$, then $u_iK[Z_i]\subset x_1S$ and the previous equality holds. If $x_1\nmid u_i$, then $x_1\in Z_i$ and $LCM(u_i,x_1)=x_1u_i$. Obviously, we get $x_1u_iK[Z_i] \subset u_i K[Z_i] \cap x_1 S$. For the other inclusion, chose $v\in u_i K[Z_i] \cap x_1 S$ a monomial. It follows that $v\in u_iK[Z_i]$ and $x_1|v$ and thus $x_1u_i|v$, since $x_1\nmid u_i$. Therefore, $v\in x_1u_iK[Z_i]$ and we are done.

By our assumption that $(I:x_1)\subsetneq (J:x_1)$, there exists some $i$ such that $u_i K[Z_i] \cap x_1 S\neq \{0\}$.
Thus, we obtain a Stanley decomposition for $(J/I)\cap(x_1)$ with its Stanley depth $\geq$ than the Stanley depth of the given decomposition for $J/I$.
\end{proof}

\vspace{2mm} \noindent {\footnotesize
\begin{minipage}[b]{15cm}
 Mircea Cimpoea\c s, Simion Stoilow Institute of Mathematics, Research unit 5, P.O.Box 1-764, Bucharest 014700, Romania\\
 E-mail: mircea.cimpoeas@imar.ro
\end{minipage}}


\begin{thebibliography}{99}
\bibitem[1]{par}Csaba Biro, David M.Howard, Mitchel T.Keller, William T.Trotter, Stephen J.Young 
"Interval partitions and Stanley depth", Preprint 2008.                 
\bibitem[2]{mirc} M.\ Cimpoeas, \textit{Stanley depth for monomial complete intersection}, Bull. Math.
              Soc. Sc. Math. Roumanie \textbf{51(99), no.3}, (2008), 205-211.
\bibitem[3]{mirci} M.\ Cimpoeas, \textit{Several inequalities regarding Stanley depth}, Romanian Journal of Math. and Computer Science \textbf{2(1)}, (2012), 28-40.              
              
 
\bibitem[4]{hvz} J.\ Herzog, M.\ Vladoiu, X.\ Zheng, \textit{How to compute the Stanley depth of a monomial ideal},
           Journal of Algebra \textbf{322(9)}, (2009), 3151-3169.

\bibitem[5]{okazaki} R.\ Okazaki, K.\ Yanagawa, Alexander duality and Stanley decomposition of multi-graded modules, Journal of Algebra \textbf{340}, (2011), 32-52.

\bibitem[6]{pop} D.\ Popescu, \textit{An inequality between depth and Stanley depth}, http://arxiv.org/pdf/0905.4597.pdf, Preprint 2010.       
\bibitem[7]{asia1} A. Rauf, \textit{Stanley Decompositions, Pretty Clean Filtrations and Reductions Modulo Regular Elements}, Bull. Math. Soc. Sc. Math. Roumanie, \textbf{50(98)}, (2007), 347-354.
\bibitem[8]{shen} Y.\ Shen, \textit{Stanley depth of complete intersection monomial ideals and upper-discrete
    partitions}, Journal of Algebra \textbf{321(2009)}, 1285-1292.
\bibitem[9]{stan} R.\ P.\ Stanley, Linear Diophantine equations and local cohomology, Invent. Math. \textbf{68}, 1982, 175-193.    
\end{thebibliography}
\end{document}